\documentclass{amsart}
\usepackage{hyperref}
\usepackage{mathrsfs}
\usepackage{amsmath}
\usepackage{amssymb}
\usepackage{amsthm}
\usepackage[all]{xy}
\usepackage{bbm}
\usepackage[utf8]{inputenc}

\newcommand{\der}{\mathrm{Der}}
\newcommand{\dder}{\mathbb{D}\mathrm{er}}

\newcommand{\jac}{\mathfrak{Jac}}
\newcommand{\kf}{\mathbbm{k}}
\renewcommand{\hom}{\mathrm{Hom}}
\renewcommand{\phi}{\varphi}
\newcommand{\dbo}{\{\hspace{-3pt}\{}
\newcommand{\dbc}{\}\hspace{-3pt}\}}
\newcommand{\nat}{\mathbb{N}}
\newcommand{\zed}{\mathbb{Z}}
\newcommand{\mder}{\mathbb{M}\mathrm{Der}}
\newcommand{\sgn}{\mathrm{sgn}}
\newcommand{\sym}{\mathfrak{S}}
\title[]{On double Poisson structures on commutative algebras}
\author[Geoffrey Powell]{Geoffrey Powell}
\address{LAREMA, UMR 6093 du CNRS et de l'Université d'Angers, Université 
Bretagne Loire, France}
\email{Geoffrey.Powell@math.cnrs.fr}
\date{}

\begin{document}

\newtheorem{thm}{Theorem}[section]
\newtheorem{prop}[thm]{Proposition}
\newtheorem{cor}[thm]{Corollary}
\newtheorem{lem}[thm]{Lemma}

\newtheorem{THM}{Theorem}
\newtheorem{COR}[THM]{Corollary}
\newtheorem{PROP}[THM]{Proposition}

\theoremstyle{definition}
\newtheorem{defn}[thm]{Definition}
\newtheorem{exam}[thm]{Example}
\newtheorem{hyp}[thm]{Hypothesis}

\theoremstyle{remark}
\newtheorem{rem}[thm]{Remark}
\newtheorem{nota}[thm]{Notation}

\begin{abstract}
Double Poisson structures (à la Van den Bergh) on commutative algebras are 
considered. The main result shows that there are no non-trivial such structures on 
polynomial algebras of Krull dimension greater than one. 
For an arbitrary commutative algebra $A$, this places significant restrictions on 
possible double Poisson structures. Exotic double Poisson structures are 
exhibited by the case of the polynomial algebra on a single generator,
 previously considered by Van den Bergh.
\end{abstract}

\maketitle

\section{Introduction}

The notion of a double Poisson structure on an associative $R$-algebra (for $R$ 
a commutative unital ring) was 
introduced by Van den Bergh as a form of non-commutative Poisson structure 
\cite{vdB}; 
the structure is defined by a double bracket, which is an $R$-linear map 
$A^{\otimes 2} \rightarrow A^{\otimes 2}$ satisfying  antisymmetry and 
non-commutative derivation conditions. For a double Poisson structure, the 
double bracket also satisfies the double Jacobi relation (see Sections 
\ref{sect:multider} and 
\ref{sect:recollect}).  

The naïve relationship with (commutative) Poisson structures is as follows: 
composing with the multiplication map of $A$ gives a bracket $A \otimes A 
\rightarrow A$ which induces a Poisson structure on the abelianization of $A$. 
More generally, a double Poisson structure induces a Poisson structure on the 
associated representation schemes \cite{vdB}. There is a related, but 
weaker, version of non-commutative Poisson structure, due to Crawley-Boevey 
\cite{CB}; this is sufficient to induce a Poisson structure on the 
 representation schemes.

The notion of double Poisson structure is very rigid; 
nevertheless, interesting  examples are known, for example those related to 
non-commutative symplectic structures. Moreover, a classification of certain 
double Poisson structures on free associative 
algebras (tensor algebras) has been given in small rank \cite{ors,sok}; 
however,  a double Poisson structure 
on a non-commutative algebra does not in general induce a double Poisson on its 
abelianization. 

It is natural to consider what happens when the algebra $A$ is already 
commutative. For example, Van den Bergh stated a classification of 
(homogeneous) double 
Poisson structures on the polynomial algebra $\kf [t]$ over a field: up to 
scalar, there are only two 
 non-trivial (homogeneous) structures. A proof of the corresponding result 
(over 
a commutative ring $R$ on which squaring is injective) is given here as 
Proposition \ref{prop:poly_dim1}, also giving the (essentially unique) 
non-homogeneous example.
This provides important {\em exotic} examples.

For higher Krull dimension the situation is more dramatic; the following is 
Theorem \ref{thm:no_dP_rank>2} below:

\begin{THM}
 Let $R$ be a commutative ring on which the squaring map $x \mapsto x^2$ is 
injective, then there is no non-trivial double Poisson structure on  $A:= 
R[t_1, \ldots, t_d]$ for  $d \geq 2$. 
\end{THM}

The result is a simple consequence of a general structure result on  
multi-derivations (see Theorem \ref{thm:multider_poly}); these 
multi-derivations 
(defined in Section \ref{sect:multider}) correspond to the $n$-brackets of Van 
den Bergh, except that the `anti-equivariance' condition with respect to the 
action of the cyclic group $\zed/n$ is not imposed.

This highlights the fact that, on commutative algebras, the axioms of a double 
Poisson structure are highly restrictive and provides further evidence that the 
notion of double Poisson structure should be relaxed, considering weaker 
structures such as Crawley-Boevey's non-commutative Poisson structures. 

Other  authors have observed that it is useful to relax the axioms of double 
Poisson algebras (see \cite{arth}, for example); it is however desirable (from 
the computational viewpoint) to retain the multi-derivation property, so the 
general structure result,Theorem  \ref{thm:multider_poly},  applies in this 
setting. Corollary \ref{cor:induced_bracket_trivial} shows that, in the 
polynomial case (of Krull dimension greater than one), 
 this relaxation is not sufficient to be able to construct non-trivial 
non-commutative 
Poisson structures (in the sense of \cite{CB}).

Section \ref{sect:comm} considers the general case of double Poisson structures 
on a commutative algebra. These are either {\em standard}, arising from double 
brackets on polynomial 
algebras, or are {\em exotic}. The results for polynomial algebras give a 
reasonable understanding of 
the standard double Poisson structures; the exotic case is illustrated by the 
results for $\kf [t]$, as indicated above. Further consequences will be 
considered elsewhere. 

Various other lines of investigation are possible. For instance, the work of 
Berest, Ramadoss et al. \cite{BKR,BCER} suggests that double Poisson structures 
for  algebras should be studied in the derived setting.

\bigskip
\noindent 
Funding: The  author was partially supported by the project {\em Nouvelle 
Équipe},  convention No. 2013-10203/10204 between the Région des Pays de la 
Loire and the Université d'Angers.

\section{Multi-derivations}
\label{sect:multider}

Fix a commutative unital ring $R$ and a unital,  associative $R$-algebra $A$; 
all tensor 
products are taken over $R$. 
For $2 \leq n \in \nat$, the symmetric group $\sym_n$ acts by place 
permutations 
on the tensor product $A^{\otimes n}$
 ($\sigma (a_1 \otimes \ldots \otimes a_n) = a_{\sigma^{-1} (1)} 
\otimes \ldots \otimes a_{\sigma^{-1}(n)}$) and, hence,  
 by conjugation on
  $
  \hom_R (A^{\otimes n} , A^{\otimes n}) 
 $ 
via $\phi \mapsto \sigma \cdot \phi := \sigma \circ \phi\circ  \sigma^{-1}$, so 
that a linear map $\phi$ is $\sym_n$-equivariant if and only if it is fixed 
under this action.
 The group $\zed/n$ is considered as a subgroup of $\sym_n$, hence the above 
action restricts to $\zed/n$.

The $R$-module of double derivations $\dder (A)$ is by definition the submodule 
\[
\der (A, A^{\otimes 2}) \subset \hom_R (A, A^{\otimes 2}) 
\]
of derivations, where $A^{\otimes 2}$ is equipped with the outer bimodule 
structure; 
 explicitly $\psi \in \hom_R (A, A^{\otimes 2})$ belongs to $\dder (A)$ if and 
only if, for all $a, b \in A$, $\psi (ab) = (a \otimes 1) \psi (b) + \psi (a)(1 
\otimes b)$, using the product in $A^{\otimes 2}$. (See \cite{Ginzburg_NC}, for 
example.)

\begin{exam}
 The double derivation $d_A \in \dder (A)$ is the $R$-linear map $a \mapsto a 
\otimes 1 - 1 \otimes a$.  This induces the universal derivation, $A 
\rightarrow 
\Omega^{nc}_A$, where the bimodule $\Omega^{nc}_A$ of non-commutative 
differentials
 is identified as the kernel of the multiplication $A \otimes A 
\stackrel{\mu}{\rightarrow} A$. 
 \end{exam}
 
\begin{lem}
\label{lem:dder_prod_codomain}
 Let $A$ be a commutative $R$-algebra. Multiplication at the codomain 
$A^{\otimes 2}$ induces a morphism of $R$-modules:
 \[
  \dder (A) \otimes A^{\otimes 2} 
  \rightarrow 
  \dder (A).
 \]
In particular, the double derivation $d_A$ gives rise to the morphism of 
$R$-modules:
\[
 \Pi : A^{\otimes 2} 
  \rightarrow 
  \dder (A)
\]
sending $\Theta \in A^{\otimes 2}$ to $a \mapsto (a \otimes 1 - 1 \otimes a) 
\Theta$.
\end{lem}

\begin{proof}
 Straightforward.
\end{proof}

\begin{rem}
 This result does not require that $A$ is a commutative and corresponds to the 
usual $A$-bimodule structure on $\dder (A)$ provided by the inner bimodule 
structure of $A^{\otimes 2}$. This formulation is given for ease of comparison with Lemma \ref{lem:mder_product} (where 
commutativity is required).
\end{rem}

By analogy with the case of double derivations,  $ \phi \in   \hom_R 
(A^{\otimes n} , A^{\otimes n}) $ is said to be a derivation with respect to 
the 
last variable if,  $\forall a, b \in A$ and $\forall \alpha \in A^{\otimes 
n-1}$:
\[
 \phi (\alpha \otimes ab) = (a\otimes 1^{\otimes n-1}) \phi (\alpha \otimes b) 
+ 
\phi (\alpha \otimes a) (1^{\otimes n-1} \otimes b), 
\]
using the product of $A^{\otimes n}$. This allows the following definition of 
multi-derivations, where the $\zed/n$-action is used to define the relevant 
bimodule structures.

\begin{defn}
For $2\leq n \in \nat$, the $R$-module of multi-derivations 
\[
\mder (A^{\otimes n}, A^{\otimes n}) \subset \hom_R (A^{\otimes n}, A^{\otimes 
n})
\]
is the submodule of morphisms $\phi$ such that $\sigma \cdot \phi $ is a 
derivation with respect to the last variable, for every $\sigma \in \zed/n$.

Let 
\[
 \mder (A^{\otimes n}, A^{\otimes n}) ^{\sgn} \subset \mder (A^{\otimes n}, 
A^{\otimes n})
\]
denote the sub $R$-module of multi-derivations $\phi$ such that  $\sigma \cdot 
\phi = (-1)^{\sgn(\sigma)} \phi$, 
 $\forall \sigma \in \zed/ n$.
\end{defn}

The following is clear from the definition: 

\begin{lem}
\label{lem:Z/n_stability}
 The sub $R$-modules
\[
\mder (A^{\otimes n}, A^{\otimes n}) ^{\sgn} \subset
\mder (A^{\otimes n}, A^{\otimes n}) \subset \hom_R (A^{\otimes n}, A^{\otimes 
n})
\]
are stable under the action of $\zed/n$.
\end{lem}

\begin{rem}
\label{rem:n_bracket}
For $2 \leq n \in \nat$, $\mder (A^{\otimes n}, A^{\otimes n}) ^{\sgn} $ is the 
$R$-module of $n$-brackets (in the terminology of 
 \cite[Definition 2.2.1]{vdB}). In particular, for  $n=2$, this 
gives the definition of a double bracket, namely an {\em anti-symmetric} 
bi-derivation and, for $n=3$, 
triple brackets are multi-derivations which are cyclically invariant.
\end{rem}

The following results provide analogues of Lemma \ref{lem:dder_prod_codomain}.

\begin{lem}
\label{lem:mder_product}
 Let $A$ be a commutative $R$-algebra and $2 \leq n\in \nat$. Multiplication in 
the codomain induces a morphism of  $R[\zed/n]$-modules  
 \[
  \mder (A^{\otimes n}, A^{\otimes n}) \otimes A^{\otimes n} \rightarrow   
\mder 
(A^{\otimes n}, A^{\otimes n}),
 \]
where the left hand side is equipped with the diagonal $\zed/n$-action.
\end{lem}

\begin{proof}
 Straightforward.
\end{proof}

\begin{prop}
\label{prop:phin_morphism}
 Let $A$ be a commutative $R$-algebra and $2 \leq n\in \nat$. The map $\phi_n 
\in \hom 
_R (A^{\otimes n}, A^{\otimes n})$ defined by 
 \[
  \phi_n (a_1 \otimes \ldots \otimes a_n) 
  := 
  \prod _{\sigma \in \zed/n} \sigma (a_{\sigma(n)} \otimes 1^{\otimes n-1} - 
1^{\otimes n-1} \otimes a_{\sigma(n)}) 
 \]
(where the product is formed in $A^{\otimes n}$) is a $\zed/n$-equivariant 
multi-derivation (that is 
$
 \phi_n \in \mder (A^{\otimes n}, A^{\otimes n})^{\zed/n}
$).

In particular, $\phi_n$ together with the map of Lemma \ref{lem:mder_product} 
induce a $\zed/n$-equivariant map:
\[
 \Pi_n : A^{\otimes n} \rightarrow \mder (A^{\otimes n}, A^{\otimes n}).
\]
\end{prop}

\begin{proof}
 That $\phi_n$ is $\zed/n$-equivariant is clear from the construction. The 
proof 
that it is a multi-derivation is analogous to the proof that  $d_A$ is a double 
derivation. 
 The final statement is then an immediate consequence of Lemma 
\ref{lem:mder_product}.
\end{proof}

\begin{rem}
\label{rem:phi2_phi3}
 For $n =2$, $\phi_2 (x \otimes y) = (x \otimes 1 - 1 \otimes x)(y \otimes 1 - 
1 
\otimes y)$ is clearly invariant under exchange of $x$ and $y$. This reflects 
the fact that, when $A$ is commutative, $A^{\otimes 2}$ has a canonical 
bimodule 
structure given by 
 the algebra structure. 
 
 For $n=3$, $\phi_{a,b,c} := \phi_3(a \otimes b \otimes c)$ is given explicitly 
by: 
  \begin{eqnarray*}
  \phi_{a,b,c} &=&
  (c \otimes 1 \otimes 1 - 1 \otimes 1 \otimes c) 
(1 \otimes 1 \otimes b - 1 \otimes b \otimes 1) 
( 1 \otimes a \otimes 1 - a \otimes 1 \otimes 1)
\\
&=&
ac \otimes b \otimes 1 - ac \otimes 1 \otimes b  
-c \otimes ab \otimes 1 + c \otimes a \otimes b - a  \otimes b \otimes c 
\\&&
+ a \otimes 1
\otimes bc  + 1 \otimes ab \otimes c - 1 \otimes a \otimes bc,
  \end{eqnarray*}
where the terms have been arranged using the left lexicographical order for the 
partial order corresponding to the number of terms in a monomial in $a, b,c$ of 
$A$. (Observe that there is a unique term of maximal lexicographical order, 
namely $ac \otimes b \otimes 1$.)

The  expression for $ \phi_3 (a \otimes b \otimes c)$ is normalized (up to 
sign) 
by the choice of the bimodule structure of $A^{\otimes 3}$, corresponding to 
the 
factor $(c \otimes 1 \otimes 1 - 1 \otimes 1 \otimes c)$. Although $\phi_3$ is 
invariant under the action of $\zed/3$, $\phi_3$ evaluated on 
$b \otimes c \otimes a$ clearly gives a different expression, contrary to the 
behaviour for $n=2$.
\end{rem}

\subsection{Graded algebras}
\label{subsect:graded_mder}

When $A$ is an $\zed$-graded $R$-algebra it is natural to consider the graded 
components of multi-derivations.

\begin{rem}
The grading is not taken into account in the symmetric monoidal structure on 
graded $R$-modules.
\end{rem}

\begin{lem}
\label{lem:grading_mder}
For $A$  a $\zed$-graded $R$-algebra which is finitely-generated as a graded 
algebra and $2 \leq n \in \nat$, there is a $R[\zed/n]$-equivariant 
decomposition into homogeneous components:
 \[
  \mder (A^{\otimes n}, A^{\otimes n}) \cong \bigoplus_{t} \mder (A^{\otimes 
n}, 
A^{\otimes n})^t
 \]
where $\mder (A^{\otimes n}, A^{\otimes n})^t$ is the submodule of morphisms of 
degree $t$.

In particular, any element $\phi \in   \mder (A^{\otimes n}, A^{\otimes n})$ 
can 
be written in terms of homogeneous components $\phi = \sum _{t \in \zed} \phi 
^t$, where $\phi^t = 0$ for $|t|\gg 0$. 
\end{lem}

\begin{proof}
The multi-derivation property and the fact that $A$ is assumed to be 
finitely-generated implies that an element $\phi \in \mder (A^{\otimes n}, 
A^{\otimes n}) $ is determined by its restriction to $(V)^{\otimes n}$ for $V$ 
a 
finitely-generated graded $R$-submodule of $A$, 
so that $(V)^{\otimes n}$ is a finitely-generated $R$-module. The proof is then 
straightforward. 
\end{proof}

\begin{nota}
\label{nota:grading_mder}
 For $A, \phi \neq 0$ as in Lemma \ref{lem:grading_mder}, write $\phi^{\min}$ 
and $\phi^{\max}$ respectively for the non-trivial homogeneous components of 
minimal and maximal degrees. 
\end{nota}

\section{Recollections on double Poisson algebras}
\label{sect:recollect}

\begin{defn}
\label{defn:double_Jacobi}
 For $\phi \in \hom_R (A^{\otimes 2}, A^{\otimes 2})$, the 
double Jacobiator 
$\jac (\phi)$
is 
\[
 \jac (\phi) := \sum_{\sigma \in \zed/3} 
\sigma \cdot \big( 
(\phi \otimes 1_A ) \circ (1_A \otimes \phi)
\big) 
\in \hom_R (A^{\otimes 3}, A^{\otimes 3})^{\zed/3}.
\]
\end{defn}

A basic fact is the following (using the terminology of $n$-brackets recalled 
in 
Remark \ref{rem:n_bracket}).

\begin{prop}
\label{prop:jac_double_is_triple}
\cite[Proposition 2.3.1]{vdB}
If $\phi$ is a double bracket on $A$, then    $\jac (\phi)$ is a triple bracket.
\end{prop}

\begin{defn}
 \cite{vdB}
A double Poisson structure on $A$ is a double bracket $\dbo, \dbc : 
A\otimes A \rightarrow A \otimes A$, $a \otimes b \mapsto \dbo a, b \dbc$, 
such that the double Jacobiator  $\dbo , , \dbc := \jac \big(\dbo, \dbc \big)$
is zero (the {\em double Jacobi relation}).
\end{defn}

\begin{rem}
 For $A$, $\dbo, \dbc$ a double Poisson algebra, the bracket 
 \[
  \{ , \} : A^{\otimes 2} \rightarrow A 
 \]
defined as the composite of $\dbo, \dbc$ with the product of $A$ is a left 
Leibniz algebra, by \cite[Corollary 2.4.4]{vdB}. Moreover, 
  \cite[Proposition 1.4]{vdB} implies  that, if $A$ is 
commutative, $\{ , \}$ defines a Poisson algebra structure on $A$.
\end{rem}

\subsection{Graded algebras}

As in Section \ref{subsect:graded_mder}, let  $A$ be  a $\zed$-graded 
$R$-algebra which is finitely-generated as a graded algebra.

\begin{defn}
A double Poisson structure $\dbo, \dbc$  on  the graded  algebra $A$  is 
homogeneous if $\dbo, \dbc = \dbo , \dbc ^t$ for some $ t \in \zed$.
\end{defn}

As in Notation \ref{nota:grading_mder}, the following notation is adopted:

\begin{nota}
 For $A$ as above and $\dbo, \dbc$ a double Poisson structure on $A$,  write  
$\dbo, \dbc ^{\min} $ and $\dbo , 
\dbc ^{\max}$ for the components of  minimal (respectively maximal) degree.
\end{nota}

\begin{lem}
\label{lem:homog_dP}
For $A$, $\dbo, \dbc$ as above, $\dbo, \dbc ^{\min} $ and $\dbo , \dbc ^{\max}$ 
define homogeneous 
double Poisson structures on $A$.
\end{lem}

\begin{proof}
 Straightforward.
\end{proof}

\section{Multi-derivations for polynomial algebras}

In this section, $A$ is taken to be the polynomial algebra $R[t_1, \ldots 
,t_d]$, where $d \geq 2$ and $R$ is a commutative unital ring. Hence $A ^{\otimes 2}$ 
is a polynomial algebra on $2d$ generators, and the elements $t_i \otimes 1 - 1 
\otimes t_i$ are algebraically independent and can be extended to a set of 
algebra generators of $A^{\otimes 2}$. In particular, the elements $t_i \otimes 
1 -  1 \otimes t_i$ are regular elements of $A ^{\otimes 2}$.

A key observation is the following:

\begin{lem}
\label{lem:dder_poly}
 Let $A$  be the polynomial algebra $R[t_1, \ldots ,t_d]$, where  $d \geq 2$. 
The morphism of $R$-modules of Lemma \ref{lem:dder_prod_codomain}
  \[
 \Pi :  A^{\otimes 2} \rightarrow \dder (A)
 \]
is an isomorphism.
\end{lem}

\begin{proof}
It is straightforward to see that $\Pi$ is a monomorphism of $R$-modules, hence 
it suffices to check surjectivity.
 Consider $\phi \in \dder (A)$ and $x, y \in A$, the commutativity relation $xy 
= yx$ gives in $A^{\otimes 2}$
 \[
  \phi (xy) = (x\otimes 1) \phi (y) + \phi (x) (1 \otimes y) = (y \otimes 1) 
\phi (x) + \phi (y) (1 \otimes x),  
 \]
thus
\begin{eqnarray}
\label{eqn:fund}
 (x \otimes 1 - 1 \otimes x) \phi (y) = (y \otimes 1- 1 \otimes y) \phi (x).
\end{eqnarray}
The double derivation property of $\phi$ implies that $\phi$ is determined by 
its restriction to the $R$-module generated by the $t_i$ and, by $R$-linearity, 
by the elements $\phi (t_i)$, $i \in \{1, \ldots ,d\}$. 
Taking $x= t_i$ and $y = t_j$ for $i \neq j$ (recall that $d \geq 2$ by 
hypothesis), equation (\ref{eqn:fund}) implies that 
\[
 \phi(t_\alpha ) = (t_\alpha \otimes 1 - 1 \otimes t_\alpha ) \Theta
\]
for $\alpha \in \{i, j \}$ and for some $\Theta \in A^{\otimes 2}$. Hence this 
equation holds for all $\alpha \in \{1, \ldots ,d\}$, showing that $\phi = 
\Pi(\Theta) $, as required.
\end{proof}

\begin{rem}
\label{rem:gen_dder}
The above argument extends to treat the map 
\[
\Pi_M :  M \rightarrow \der (A, M),
\]
when $M$ is a free $A^{\otimes 2}$-module (the module structure giving the 
$A$-bimodule structure). 
\end{rem}

\begin{thm}
 \label{thm:multider_poly}
 Let $A$ be the polynomial algebra $R[t_1, \ldots ,t_d]$, where  $d \geq 2$.
For $2 \leq n \in \nat $ the morphism of $R$-modules of Proposition 
\ref{prop:phin_morphism}, 
 \begin{eqnarray*}
\Pi_n : A ^{\otimes n} & \rightarrow & \mder ( A ^{\otimes n},  A ^{\otimes n}), 
  \end{eqnarray*}
is an isomorphism of $R[\zed/n]$-modules.
\end{thm}

\begin{proof}
 It clearly suffices to prove that $\Pi_n$ is an isomorphism of $R$-modules. It 
is straightforward to check that $\Pi_n$ is a monomorphism, thus it suffices to 
show that $\Pi_n$ is surjective. 
 
First consider the case $n=2$ and take $\phi \in \mder (A^{\otimes 2}, 
A^{\otimes 2})$; then, for fixed $a 
\in A$, the map $\phi (a \otimes - ) : A\rightarrow A^{\otimes 2}$ belongs to 
$\dder (A)$, hence Lemma \ref{lem:dder_poly} implies that 
 \[
  \phi (a \otimes b) = (b \otimes 1 - 1 \otimes b) \Theta_a 
 \]
for some $\Theta_a \in A^{\otimes 2}$ that is independent of $b$. 
 
Now take $b=t_1$, so that $b \otimes 1 - 1 \otimes b$ is a regular element of 
$A^{\otimes 2}$. 
It follows that  $a \mapsto \Theta_a$ defines a double derivation of $\dder 
(A)$. Again by  Lemma \ref{lem:dder_poly}, $\Theta_a$ can be written as $(a 
\otimes 1 -1 \otimes a) \Theta$ , for some $\Theta \in A^{\otimes 2}$ that is 
independent of $a$, so that 
\[
 \phi (a \otimes b ) = (a \otimes 1 -1 \otimes a)(b \otimes 1 - 1 \otimes b)  
\Theta
\]
for {\em any} $a, b \in A$, as required.

For  $n >2$, the above argument is modified in the obvious way, by appealing 
to Remark \ref{rem:gen_dder}. For example, given $\phi \in \mder (A^{\otimes 
n}, A^{\otimes n})$, fix $\alpha \in A^{\otimes n-1}$ and consider the map 
$\phi (\alpha \otimes - )$ as belonging to $\der (A, A^{\otimes n})$, where 
$A^{\otimes n}$ is the free bimodule with respect to the outer bimodule 
structure. As above, one deduces that 
 \[
  \phi (\alpha  \otimes b) = (b \otimes 1^{\otimes n-1} - 1^{\otimes n-1} 
\otimes b) \Theta_\alpha 
 \]
where $\Theta_\alpha$ is independent of $b$. The argument is then repeated 
recursively, starting as above by analysing $\Theta_{\alpha}$, at each step 
reducing the number of dependencies. 
\end{proof}

\begin{rem}
 The argument for the case $n=2$ (and, by extension, the general case) depends 
 on the fact that each  $t_i \otimes 1- 1 \otimes t_i$ is a regular 
element. Clearly the argument fails in general for $A$ an arbitrary commutative 
ring; even the injectivity of $\Pi_n$ need not hold. 
\end{rem}

\begin{exam}
\label{exam:exotic}
 For $A = \kf [t]$, with $\kf$ a field, there is a double bracket defined by 
 \[
  t \otimes t \mapsto t \otimes 1 - 1 \otimes t
 \]
(see Section \ref{sect:symm}). This is clearly not in the image of $\Pi_2$.
\end{exam}

\begin{rem}
 For the free associative algebra $T (V)$ on a free $R$-module $V$ and $2\leq n 
\in \nat$, any morphism $V^{\otimes n} \rightarrow T(V)^{\otimes n}$ extends 
uniquely to an element of $\mder (T(V)^{\otimes n}, T(V)^{\otimes n})$ (and 
clearly every multi-derivation is determined by its restriction to $V^{\otimes 
n}$). 
The corresponding result is false in the commutative case; Theorem 
\ref{thm:multider_poly} provides an analogous (but much stronger) result.  
\end{rem}

For $A$ a commutative $R$-algebra, the multiplication $\mu : A^{\otimes2} 
\rightarrow A$ induces an $R$-linear map 
$
 \hom_R (A^{\otimes 2} , A^{\otimes 2} ) 
 \rightarrow 
  \hom_R (A^{\otimes 2} , A) 
$
which restricts to a map $ \mder (A^{\otimes 2} , A^{\otimes 2} ) 
 \rightarrow 
  \hom_R (A^{\otimes 2} , A)$.

\begin{cor}
\label{cor:induced_bracket_trivial}
  Let $A$ be the polynomial algebra $R[t_1, \ldots ,t_d]$, where  $d \geq 2$.
  Then the morphism of $R$-modules 
  \[
   \mder (A^{\otimes 2} , A^{\otimes 2} ) 
 \rightarrow 
  \hom_R (A^{\otimes 2} , A)
  \]
is trivial.
\end{cor}

\begin{proof} 
 By inspection, the map $\phi_2$ is sent to zero, whence the result, by Theorem 
\ref{thm:multider_poly}, using the definition of $\Pi_2$.
\end{proof}

\begin{rem}
 Corollary \ref{cor:induced_bracket_trivial} shows that no Poisson structure on 
$R[t_1, \ldots ,t_d]$ is induced by a multi-derivation. This shows 
 that the weakening of the notion of double Poisson structure proposed in 
\cite{arth} does not provide further non-trivial examples of non-commutative 
Poisson structures 
 (in the sense of \cite{CB}).
\end{rem}

\section{Double Poisson structures on polynomial algebras}

Let $R$ be a commutative unital ring and $A$ be a commutative $R$-algebra.

\begin{nota}
 Write 
\begin{enumerate}
 \item 
 $\Lambda^2 (A) \subset A^{\otimes 2}$ for the sub $R$-module of 
anti-commutative 
elements 
(namely the kernel of $\mathrm{id} + \tau : A^{\otimes 2} \circlearrowleft$, 
where $\tau$ transposes the tensor factors);
\item 
$(A^{\otimes 3})^{\zed/3} \subset A^{\otimes 3}$ for the sub $R$-module of 
cyclically invariant elements.
\end{enumerate}
\end{nota}

 \begin{rem}
  In characteristic two the above does not give the usual definition of 
$\Lambda^2 (A)$.
 \end{rem}

\begin{prop}
\label{prop:double_triple} 
Let $A = R[t_1, \ldots, t_d]$, where $d \geq 2$. 
\begin{enumerate}
 \item 
 The isomorphism $\Pi_2 : A^{\otimes 2} \stackrel{\cong}{\rightarrow} \mder 
(A^{\otimes 2} , A^{\otimes 2}) $ 
 restricts to  an isomorphism of  $R$-modules
 \[
  \Lambda^2 (A)
  \stackrel{\cong}{\rightarrow} 
  \mder (A^{\otimes 2} , A^{\otimes 2})^{\sgn},
 \]
 where $\mder (A^{\otimes 2} , A^{\otimes 2})^{\sgn}$ is the $R$-module of 
double brackets on $A$.
 \item 
 The isomorphism $\Pi_3 : A^{\otimes 3} \stackrel{\cong}{\rightarrow} \mder 
(A^{\otimes 3} , A^{\otimes 3}) $ 
 restricts to an isomorphism of  $R$-modules
 \[
 (A^{\otimes 3}) ^{\zed/3}
   \stackrel{\cong}{\rightarrow} 
  \mder (A^{\otimes 3} , A^{\otimes 3})^{\sgn},
 \]
 where $ \mder (A^{\otimes 3} , A^{\otimes 3})^{\sgn}$ is the $R$-module of 
triple brackets on $A$.
\end{enumerate}
\end{prop}

\begin{proof}
 An immediate consequence of the definitions and  Theorem 
\ref{thm:multider_poly}, using the $R [\zed/n]$-equivariance in the cases $n=2$ 
and $n=3$.
\end{proof}

\begin{rem}
Explicitly, for $\Psi \in \Lambda^2 (A) \subset A^{\otimes 2}$ (so that $\tau 
\Psi = - \Psi$), the associated double bracket is 
\[
 \dbo a , b \dbc _{\Psi}
 = 
 (a \otimes 1 - 1\otimes a) (b \otimes 1 - 1 \otimes b) \Psi.
\]
\end{rem}

\begin{nota}
 Let $(-)_{23}$ denote the $R$-linear map $A^{\otimes 2} \rightarrow A^{\otimes 
3}$, $a \otimes b \mapsto 1 \otimes a \otimes b$ and $(-)_{13}$ the map $a 
\otimes b \mapsto a \otimes 1 \otimes b$. 
\end{nota}

\begin{prop}
\label{prop:jac_poly}
Let $A = R[t_1, \ldots, t_d]$, where $d \geq 2$. 
 Under the isomorphisms of Proposition \ref{prop:double_triple}, the set map 
induced by the double Jacobiator 
 \[
  \Big\{ 
  \text{double brackets on $A$}
  \Big\}
  \stackrel{\jac}{\rightarrow}
   \Big\{ 
  \text{triple brackets on $A$}
  \Big\}  
 \]
 (cf Proposition \ref{prop:jac_double_is_triple})  identifies with the 
(non-linear) map 
 \begin{eqnarray*}
\mathfrak{J} :   \Lambda^2 (A) & \rightarrow &  (A^{\otimes 3}) ^{\zed/3} \\
  \Psi & \mapsto & 
  \sum _{\sigma \in \zed/3} 
  \sigma \cdot (\Psi_{13} \Psi_{23}) 
 \end{eqnarray*}
where the product $\Psi_{13} \Psi_{23}$ is formed in $A^{\otimes 3}$.
\end{prop}

\begin{proof}
 By Proposition \ref{prop:double_triple}, it suffices to identify $\jac \dbo , 
\dbc_\Psi$ in the image of $\Pi_3$. It is  clear that the expression must 
be a $\zed/3$-invariant quadratic expression in $\Psi$. The result follows  
by  direct calculation, using the anti-symmetry $\tau \Psi = - \Psi$. (The 
calculation may be simplified by using the proof 
of \cite[Proposition 2.3.1]{vdB}.) 
\end{proof}

\begin{thm}
\label{thm:no_dP_rank>2}
Let $R$ be a commutative ring on which the squaring map $x \mapsto x^2$ is 
injective, then there is no non-trivial double Poisson structure on  $A:= 
R[t_1, 
\ldots, t_d]$ for  $d \geq 2$. 
\end{thm}

\begin{proof}
 It follows from the identification given in Remark \ref{rem:phi2_phi3} that, 
for any $i, j,k \in \{ 1 , \ldots, d \}$, the element 
 $\phi_3 (t_i, t_j, t_k)$ is a regular element of $A^{\otimes 3}$ (this does 
not 
require the hypothesis upon $R$).
Hence, by Proposition \ref{prop:jac_poly},  to prove the result it suffices to 
show that $\Psi \in A^{\otimes 2}$ is zero if and only if  $ \sum _{\sigma \in 
\zed/3} 
  \sigma \cdot (\Psi_{13} \Psi_{23}) \in A^{\otimes 3}$ is zero (anti-symmetry 
of $\Psi$ plays no rôle here).

 The latter fact is seen by exploiting the natural grading of $A$ (placing the 
generators in degree one, so the grading coincides with the length grading), 
together with the induced
 left lexicographical ordering on $A^{\otimes 2}$ and $A^{\otimes 3}$. Namely, 
if $\Psi$ is non-zero, the terms of maximal lexicographical order in $\Psi$ 
contribute to 
 a non-zero term of maximal lexicographical order in $\sum _{\sigma \in \zed/3} 
  \sigma \cdot (\Psi_{13} \Psi_{23})$ (Cf. Remark \ref{rem:phi2_phi3}).
  
  Explicitly, writing $\Psi = \sum_m \alpha_m\otimes m$ in terms of the 
monomial 
basis of $A$, one considers the contributions  
  \[
   (\alpha_m)^2 \otimes m \otimes m 
  \]
in $\Psi_{13} \Psi_{23}$ to the terms of maximal lexicographical order in   
$\sum _{\sigma \in \zed/3} 
  \sigma \cdot (\Psi_{13} \Psi_{23})$. Finally, the hypothesis upon $R$ implies 
that   $(\alpha_m)^2 \otimes m \otimes m$ is zero if and only if $\alpha_m $ is 
zero.
  \end{proof}

\section{Double Poisson structures on commutative algebras}
\label{sect:comm}

In this section, $A$ denotes a commutative $R$-algebra. 

\begin{defn}
 A double bracket on $A$ is {\em standard} if it lies in the image of the 
morphism of $R$-modules
 \[
  \Pi_2 : \Lambda^2 (A) \rightarrow 
  \mder (A^{\otimes 2} , A^{\otimes 2})^{\sgn}
 \]
induced by $\Pi_2$ (as in Proposition \ref{prop:double_triple}) and is {\em 
exotic} otherwise, so that 
the $R$-module of exotic double brackets is the cokernel of the above morphism.
\end{defn}

\begin{rem}
\ 
\begin{enumerate}
 \item 
Exotic double brackets exist: cf. Example \ref{exam:exotic}. However, these 
cannot be classified easily  (cf. the case $A= R[t]$ in Section 
\ref{sect:symm}).
\item 
The restriction to $\Lambda^2 (A)$ is not  severe. For example, if $2$ is 
invertible in $R$, the inclusion $\Lambda^2 (A) \hookrightarrow A^{\otimes 2}$ 
admits the retract $ x \otimes y \mapsto \frac{1}{2} (x \otimes y - y \otimes 
x)$. 
\end{enumerate}
\end{rem}

Observe that the set map $\mathfrak{J} :   \Lambda^2 (A)  \rightarrow  
(A^{\otimes 
3}) ^{\zed/3}$ 
of Proposition \ref{prop:jac_poly} can be defined for any commutative algebra 
$A$. 

\begin{thm}
\label{thm:standard_dP}
 Let $\Psi \in \Lambda^2 (A)$ and consider the associated (standard) double 
bracket $\dbo , \dbc _\Psi := \Pi_2 (\Psi)$. Then:
 \begin{enumerate}
  \item 
  the associated bracket $\{, \} : A^{\otimes 2} \rightarrow A$ is trivial; 
 \item 
 $\dbo , \dbc _\Psi$ defines a double Poisson structure on $A$ if and only if 
$\Pi_3( \mathfrak{J}(\Psi))$ is zero in $\mder (A^{\otimes 3}, A^{\otimes 
3})^{\sgn}$.  
 \end{enumerate}
\end{thm}

\begin{proof}
 The first statement follows as for Corollary 
\ref{cor:induced_bracket_trivial}. 
The fact that $ \dbo , \dbc _\Psi$ is a standard double bracket implies that 
the 
calculation of Proposition \ref{prop:jac_poly} is universal, the only 
difference 
being 
 that the triple brackets on $A$ cannot be identified with $(A^{\otimes 
3})^{\zed/3}$ via $\Pi_3$. 
\end{proof}

\begin{rem}
 Theorem \ref{thm:standard_dP} provides a recipe for constructing examples of 
non-trivial double Poisson structures on commutative algebras: for any $\Psi 
\in 
\Lambda ^2 (R[t_i])$ ($R[t_i]$ a polynomial algebra) it suffices to pass to a 
quotient $A$ of $R[t_i]$ for which  $\Pi_3( \mathfrak{J}(\Psi))$ is trivial in  
$\mder 
(A^{\otimes 3}, A^{\otimes 3})^{\sgn}$. Note that, in all cases, the associated 
bracket (as in Corollary \ref{cor:induced_bracket_trivial}) is trivial. 
\end{rem}

\appendix
\section{Double Poisson structures on $R[t]$}
\label{sect:symm}

In \cite[Example 2.3.3]{vdB}, Van den Bergh stated a classification of the  
(homogeneous) 
double Poisson structures on the polynomial algebra $\kf[t]$, for $\kf$ a 
field. A proof over a more general ring, also considering non-homogeneous 
structures, is given here.

\begin{prop}
\label{prop:poly_dim1}
Let $R$ be a commutative ring on which $x \mapsto x^2$ is injective, then the 
only homogeneous double Poisson structures on $A := R [t]$ are 
scalar multiples of the double Poisson brackets determined by  
 \begin{eqnarray*}
  \dbo t , t \dbc^{1} & = & t \otimes 1 - 1 \otimes t \\
  \dbo t , t \dbc^3 & = & t^2  \otimes t - t \otimes t^2 = (t \otimes 1 - 1 
\otimes t) (t \otimes t),
 \end{eqnarray*}
 where the suffix corresponds to the degree of the element $\dbo t , t\dbc$.
 
In general, for $\lambda, \mu , \nu \in R$, 
\[
 \dbo t , t \dbc = \lambda  \dbo t , t \dbc^{1} + \mu (t^2\otimes 1 - 1 \otimes 
t^2) + \nu \dbo t ,t \dbc ^3
\]
defines a double Poisson structure if and only if $\lambda \nu - \mu^2 =0$ and 
any double Poisson structure on $A$ is of this form.
\end{prop}

\begin{proof} 
A double Poisson structure on $R[t]$ is determined by $\dbo t, t \dbc$. 
It is straightforward to verify that   $\dbo t , t \dbc^{1}$ and $ \dbo t , t 
\dbc^3$ define homogeneous 
double Poisson structures on $R [t]$.

The  derivation property (using  induction upon $n \geq 1$) implies 
that 
\[
 \dbo t , t^n \dbc = \Big( \sum _{i+j = n-1} t^i \otimes t^j \Big) \dbo t,t 
\dbc, 
\]
where the product  is formed in the algebra $A^{\otimes 2}$.

Anti-symmetry implies that a  homogeneous double bracket $\dbo , \dbc ^N$ is an 
$R$-linear combination of terms of the form $(t^{N-i} \otimes t^i  - t^i 
\otimes 
t^{N-i})$, for  $N$ corresponding to the homogeneous degree and $0 \leq i < 
N/2$.

First consider the case where  $\dbo t, t \dbc ^N = \lambda (t^{N-i} \otimes 
t^i 
 - t^i \otimes t^{N-i})$ for some $i$; this is subdivided into two cases:
\begin{enumerate}
 \item 
 $\dbo t, t\dbc ^N = \lambda (t^{a+1} \otimes t^a - t^a \otimes t^{a+1})$, for 
$a \in \nat$ (so that $N= 2a +1$)
 and $\lambda \in R$. The cases $ a \in \{0, 1 \}$ correspond to the two 
cases given above, hence suppose that 
 $a >1$ (which implies that $2a > a+1$). 
 
Consider the coefficient of $t^{2a} \otimes 
t^{a+1} \otimes t^a$ in $\dbo t, t, t\dbc$. Write $\Phi$ for the element $(\dbo 
, \dbc \otimes \mathrm{id}) \circ (\mathrm{id} \otimes \dbo , \dbc ) (t \otimes 
t \otimes t)$.
Thus the double Jacobiator $\dbo t, t,t \dbc$ is the sum of the cyclic 
permutations of $\Phi$. 
Hence it is necessary to consider  the coefficients of  $t^{2a} \otimes t^{a+1} 
\otimes t^a$, $t ^a \otimes  
t^{2a} \otimes t^{a+1}$
 and $   t^{a+1} \otimes t^a \otimes t^{2a}$ in $\Phi$. The coefficient of the 
first is zero (the two contributions
 cancel) and the second has coefficient $- \lambda^2$; the hypothesis on $a$ 
ensures that the third cannot occur. Thus $\dbo, , \dbc =0$ implies that 
 $- \lambda^2 =0$, so that $\dbo , \dbc ^N =0$. 
 \item 
 $\dbo t, t\dbc ^N = \lambda (t^{N-a } \otimes t^a - t^a \otimes t^{N-a})$, 
with 
$N -a > a+1$. Consider the coefficient of $t^{2(N-a-1)} 
\otimes t^{a+1} \otimes t^a$ in $\dbo t, t, t\dbc$. In this case, the 
coefficient of $t^{2(N-a-1)} \otimes 
t^{a+1} \otimes t^a$ in $\Phi$ is 
 $\lambda^2$. If $2 (N-a -1) > N-a$ then the coefficients of $t^a \otimes 
t^{2(N-a-1)} \otimes 
t^{a+1}$ and $t^{a+1} \otimes t^a \otimes t^{2(N-a-1)} $ in $\Phi$ are both 
trivial. Hence (in this case) the condition 
 $\dbo, , \dbc =0$ implies that  $\lambda^2 =0$ and again $\dbo , \dbc ^N =0$.
 
 The inequality $2 (N-a -1) > N-a$ is equivalent to $N > a+2$; since $N > 2a 
+1$, by hypothesis, 
this is satisfied if $a \geq 1$  or if $a=0$ and $N>2$. In the remaining case, 
$N=2$ and $a=0$, it can be checked 
directly that $\dbo , \dbc ^N=0$.
\end{enumerate}

To complete the proof, one considers the case where $\dbo t,t \dbc^N $ has at 
least 
two non-trivial coefficients with respect to the basis $\{t^{N-i} \otimes t^i  
- 
t^i \otimes t^{N-i} | 0 \leq i < N/2 \}$.  Thus one can write 
\begin{eqnarray*}
 \dbo t,t \dbc ^N &=& \lambda (t^{N-a} \otimes t^a - t^a \otimes t^{N-a})  + 
\mu 
(t^{N-b } \otimes t^b - t^b \otimes t^{N-b}) 
\\
&&+ \sum_{b<k <N/2} \nu_k (t^{N-k} \otimes t^k - t^k \otimes t^{N-k}) 
\end{eqnarray*}
where $\lambda \neq 0$,  $\mu \neq 0$ and $0 \leq a < b < N/2$ (hence $N >2$).

Consider the coefficient of  
\[
 t ^{2N -b -a -1} \otimes t^b \otimes t^a
\]
in $\dbo t, t, t\dbc$.
As above using the notation $\Phi$, 
\begin{enumerate}
 \item 
 the coefficient of $t^{2N -b -a -1} \otimes t^{b} \otimes t^a$ in $\Phi$ is 
$\lambda^2 + \lambda \mu$;
 \item 
 the coefficient of $t^a \otimes t^{2N -b -a -1}\otimes t^{b} $ in $\Phi$ is $- 
\lambda \mu$ (the sign 
 arises from antisymmetry);
 \item 
 the term $t^b \otimes t^a \otimes t^{2N -b -a -1 }$ cannot arise in $\Phi$, 
since 
 $2N- b -a -1 > N-a$ (the difference is $N-1 -b$ and the latter is positive by 
the hypotheses). 
\end{enumerate}
It follows that the coefficient of $ t ^{2N -b -a -1} \otimes t^b \otimes t^a$ 
in $\dbo t, t, t\dbc$ is $\lambda^2$, thus $\lambda =0$, contradicting the 
hypothesis that $\lambda \neq 0$.

Finally, consider the non-homogeneous case. Here, by  Lemma \ref{lem:homog_dP}, 
the only non-trivial possibility is 
\[
 \dbo t, t \dbc = \lambda (t \otimes 1 - 1 \otimes t) + \mu  (t^2 \otimes 1 - 1 
\otimes t^2) + \nu (t^2  \otimes t - t \otimes t^2)
\]
where, if $\mu \neq 0$, then both  $\lambda$ and  $\nu$ are non zero.

The associated double Jacobiator $\dbo t, t, t \dbc $ in principle has terms in 
degrees $1$, $2$, $3$, $4$ and $5$; since $\dbo , \dbc^1$ and $\dbo , \dbc ^3$ 
give double Poisson structures, the terms in degrees $1$ and $5$ vanish (as 
already observed in Lemma  \ref{lem:homog_dP}). A straightforward calculation 
also shows that the terms in degrees $2$ and $4$ vanish. 

Finally, one finds that 
\[
 \dbo t, t, t \dbc 
 = 
 (\lambda \nu - \mu^2) \big ( \overline{1 \otimes t \otimes t^2} - \overline{1 
\otimes t^2 \otimes t}  \big ) 
\]
where $ \overline{1 \otimes t \otimes t^2}$ and $\overline{1 \otimes t^2 
\otimes 
t}$ denote the respective $\zed/3$-orbit sums. Hence the double bracket defines 
a double Poisson structure if and only if 
$\lambda \nu = \mu^2$.
\end{proof}

\begin{rem}
\ 
\begin{enumerate}
 \item 
  The transformation given by \cite[Example 2.3.3]{vdB} associated to the 
change 
of variables $t \mapsto t^{-1}$ (after extending to $\kf [t^{\pm 1}]$) 
 acts by $ \lambda \mapsto  - \nu $, $ \nu  \mapsto  - \lambda$ and  $\mu  
\mapsto  - \mu$,   as expected.
 \item 
  Over a field $\kf$, up to scalar multiplication and the action of $\kf^*$, 
considered as automorphisms of $\kf [t]$ via $\alpha : t  \mapsto \alpha t$, 
this gives the 
 single non-homogeneous example 
 \[
 \dbo t, t\dbc :=  (t \otimes 1 - 1 \otimes t) + (t^2 \otimes 1 - 1 \otimes 
t^2) 
+ (t^2  \otimes t - t \otimes t^2).
 \]
\end{enumerate}
 \end{rem}


\begin{thebibliography}{BCER12}

\bibitem[{Art}15]{arth}
S.~{Arthamonov}, \emph{{Noncommutative Inverse Scattering Method for the
  Kontsevich System}}, Letters in Mathematical Physics \textbf{105} (2015),
  1223--1251.

\bibitem[BCER12]{BCER}
Yuri Berest, Xiaojun Chen, Farkhod Eshmatov, and Ajay Ramadoss,
  \emph{Noncommutative {P}oisson structures, derived representation schemes and
  {C}alabi-{Y}au algebras}, Mathematical aspects of quantization, Contemp.
  Math., vol. 583, Amer. Math. Soc., Providence, RI, 2012, pp.~219--246.
  \MR{3013096}

\bibitem[BKR13]{BKR}
Yuri Berest, George Khachatryan, and Ajay Ramadoss, \emph{Derived
  representation schemes and cyclic homology}, Adv. Math. \textbf{245} (2013),
  625--689. \MR{3084440}

\bibitem[CB11]{CB}
William Crawley-Boevey, \emph{Poisson structures on moduli spaces of
  representations}, J. Algebra \textbf{325} (2011), 205--215. \MR{2745537}

\bibitem[{Gin}05]{Ginzburg_NC}
V.~{Ginzburg}, \emph{{Lectures on Noncommutative Geometry}}, ArXiv:0506603
  (2005).

\bibitem[ORS13]{ors}
Alexander Odesskii, Vladimir Rubtsov, and Vladimir Sokolov, \emph{Double
  {P}oisson brackets on free associative algebras}, Noncommutative birational
  geometry, representations and combinatorics, Contemp. Math., vol. 592, Amer.
  Math. Soc., Providence, RI, 2013, pp.~225--239. \MR{3087947}

\bibitem[Sok13]{sok}
V.~V. Sokolov, \emph{Classification of constant solutions of the associative
  {Y}ang-{B}axter equation on {$\rm Mat_3$}}, Theoret. and Math. Phys.
  \textbf{176} (2013), no.~3, 1156--1162, Russian version appears in Teoret.
  Mat. Fiz. {{\bf{1}}76} (2013), no. 3, 385--392. \MR{3230739}

\bibitem[VdB08]{vdB}
Michel Van~den Bergh, \emph{Double {P}oisson algebras}, Trans. Amer. Math. Soc.
  \textbf{360} (2008), no.~11, 5711--5769. \MR{2425689 (2009m:17019)}

\end{thebibliography}
\providecommand{\bysame}{\leavevmode\hbox to3em{\hrulefill}\thinspace}
\providecommand{\MR}{\relax\ifhmode\unskip\space\fi MR }
\providecommand{\MRhref}[2]{%
  \href{http://www.ams.org/mathscinet-getitem?mr=#1}{#2}
}
\providecommand{\href}[2]{#2}

\end{document}